\documentclass{article}
\usepackage{amsmath,amsfonts,amssymb,amsthm}
\usepackage{graphicx}
\title{A strong law of large numbers for scrambled net integration}
\author{Art B. Owen\\Stanford University
\and
Daniel Rudolf\\
University of Goettingen
}
\date{June 2020}

\renewcommand{\le}{\leqslant}
\renewcommand{\ge}{\geqslant}

\newcommand{\bsa}{\boldsymbol{a}}
\newcommand{\bsc}{\boldsymbol{c}}
\newcommand{\bsk}{\boldsymbol{k}}
\newcommand{\bsu}{\boldsymbol{u}}
\newcommand{\bsx}{\boldsymbol{x}}

\newcommand{\bszero}{\boldsymbol{0}}

\newcommand{\dustd}{\mathbb{U}}

\newcommand{\hk}{\mathrm{HK}}

\newcommand{\mc}{\mathrm{MC}}
\newcommand{\qmc}{\mathrm{QMC}}
\newcommand{\rqmc}{\mathrm{RQMC}}
\newcommand{\rlat}{\mathrm{RLAT}}

\newcommand{\rd}{\,\mathrm{d}}
\newcommand{\one}{\mathbf{1}}
\newcommand{\bvhk}{\mathrm{BVHK}}

\newcommand{\e}{\mathbb{E}}
\newcommand{\var}{\mathrm{var}}
\newcommand{\cov}{\mathrm{cov}}

\newcommand{\real}{\mathbb{R}}

\newcommand{\ce}{\mathcal{E}}
\newcommand{\cf}{\mathcal{F}}
\newcommand{\cn}{\mathcal{N}}
\newcommand{\ct}{\mathcal{T}}

\newtheorem{theorem}{Theorem}
\newtheorem{lemma}{Lemma}
\newtheorem{corollary}{Corollary}

\theoremstyle{definition}
\newtheorem{definition}{Definition}

\begin{document}
\maketitle

\begin{abstract}
This article provides a strong law of large numbers for integration
on digital nets randomized by a nested uniform scramble.
The motivating problem is optimization over some variables
of an integral over others, arising in Bayesian optimization.
This strong law requires that the integrand
have a finite moment of order $p$ for some $p>1$.
Previously known results implied a strong law
only for Riemann integrable functions.
Previous general weak laws of large numbers for
scrambled nets require a square integrable integrand.
We generalize from $L^2$ to $L^p$ for $p>1$ via
the Riesz-Thorin interpolation theorem.
\end{abstract}

\section{Introduction}

Numerical integration is a fundamental building block in many
applied mathematics problems.  When the integrand is a smooth
function of a low dimensional input, then classical methods
such as tensor products of Simpson's rule are very
effective \cite{davrab}.   For non-smooth integrands
or higher dimensional domains, these methods may perform poorly.
One then turns to Monte Carlo methods, where the integrand
is expressed as the expected value of a random variable
which is then sampled in a simulation and averaged.
Sample averages converge to population averages by
the law of large numbers (LLN), providing a justification
for the Monte Carlo method.

The Monte Carlo method converges very slowly to the
true answer as the number $n$ of sampled values increases.
The root mean squared error is $O(n^{-1/2})$.
Quasi-Monte Carlo (QMC) methods
\cite{dick:kuo:sloa:2013, dick:pill:2010, nied92}
replace random sampling by deterministic sampling methods. These may
be heuristically described as space filling samplers
using $n$ points  constructed to reduce the unwanted
gaps and clusters that would arise among randomly chosen inputs.
Because the inputs are not random, we cannot use the law
of large numbers to ensure that the estimate converges
to the integral as $n\to\infty$.
Such consistency is a minimal requirement of an integration
method.  For  QMC, consistency requires additional assumptions
of Riemann integrability or bounded variation, whose
descriptions we defer. Under the latter condition, the integration
error is $O(n^{-1+\epsilon})$ for any $\epsilon>0$.
QMC has proved valuable in financial valuation \cite{glasserman2003monte},
graphical rendering \cite{keller95} and solving PDEs in random environments
\cite{kuo:nuye:2016}.

In addition to knowing that a method would work
as $n\to\infty$, users also need to have some estimate
of how well it has worked for a given sample size $n$.
Monte Carlo methods make it easy to quantify
uncertainty by using the central limit theorem
in conjunction with a sample variance estimate.
Plain QMC lacks such a convenient error estimate.
Randomized QMC (RQMC) methods,
surveyed in \cite{lecu:lemi:2002},
produce random points with QMC properties.
Then a few statistically independent repeats  of the
whole RQMC process support uncertainty quantification.
One of these methods, scrambled nets \cite{rtms,owensinum},
provides estimated integrals that are consistent
as $n\to\infty$ under weaker conditions than plain QMC
requires.
It can also reduce the root mean squared error to $O(n^{-3/2+\epsilon})$
\cite{smoovar,localanti} under further conditions on the integrand.

The first panel in Figure~\ref{fig:sobol} shows $512$ MC points in the
unit square $[0,1]^2$. We see clear gaps and clumps among those points.
The second panel shows $512$ QMC points
from a Sobol' sequence described in Section~\ref{sec:snets}.
The points are very structured and fill the space quite evenly.
The third panel shows a scrambled version of those $512$
points also described in Section~\ref{sec:snets}.

\begin{figure}
\includegraphics[width=1.0\hsize]{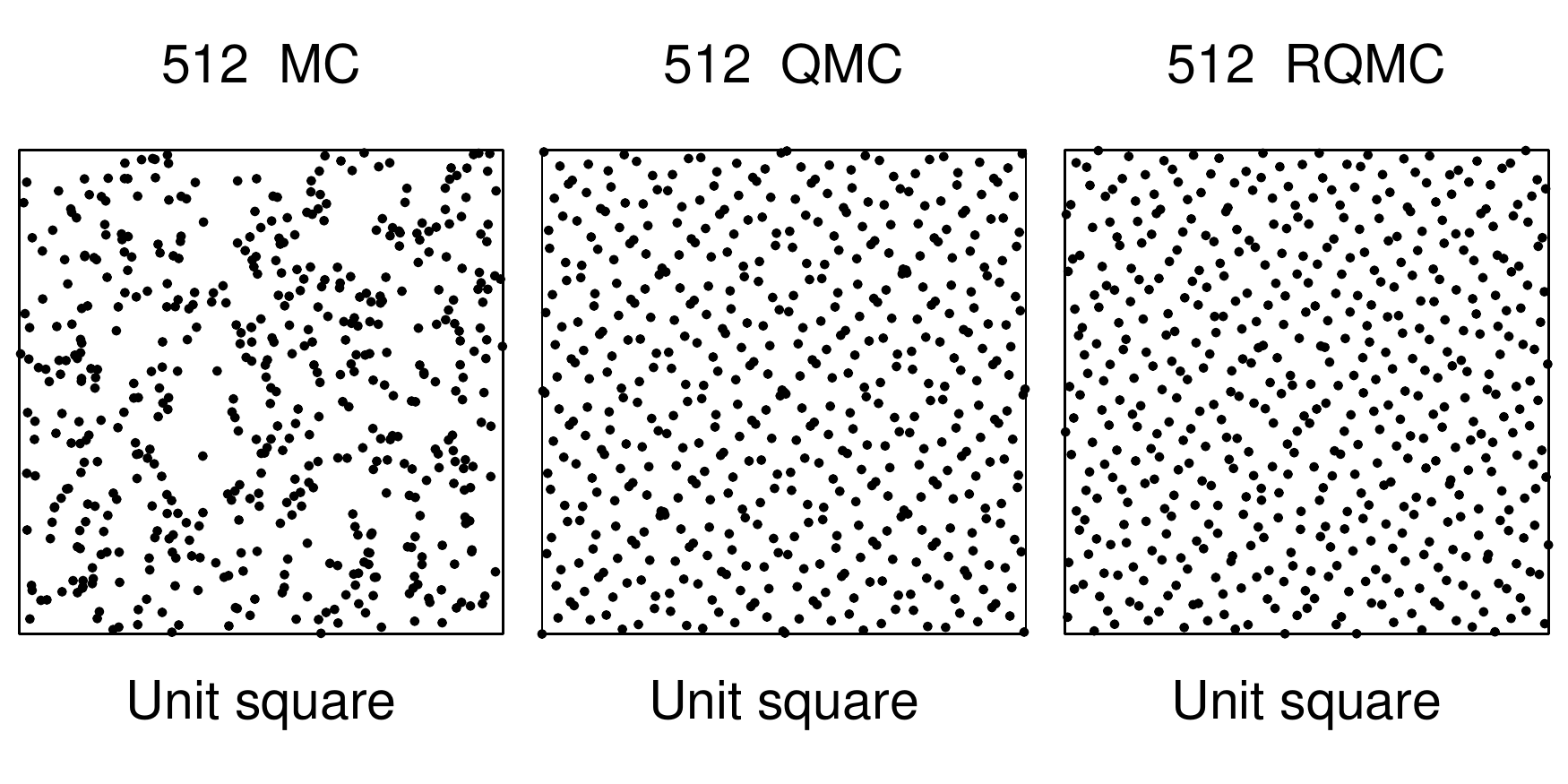}
\caption{\label{fig:sobol}
Each panel depicts $512$ points in the unit square $[0,1]^2$.
From left to right: plain Monte Carlo points, Sobol' points, scrambled
Sobol' points.
}
\end{figure}

Up to this point, we have considered the LLN
as just one result.  There are in fact strong and weak forms of
the LLN that we discuss below. The distinction does
not come up for plain Monte Carlo sampling
because both laws hold at once.
For RQMC, mostly weak laws of large numbers have been proved.
Our contribution here is to establish strong laws.
The motivation to do this comes from
the PyTorch \cite{bala:etal:2019}
tool for Bayesian optimization. A prototypical Bayesian optimization
problem is to find $\min_{\theta\in\Theta}\int_{\real^d}g(\theta,\bsx)\rd\bsx$
for some function $g(\theta,\bsx)$ and a set $\Theta$ of
allowed values for a parameter $\theta$.
In a simulation-optimization framework \cite{andr:1998}   the integral over $\bsx$
values may be approximated by a Monte Carlo average.
Integration is then a building block in a larger problem.
PyTorch has a version using RQMC points instead of MC.
Consistent estimation of the optimal $\theta$ could be proved
assuming a strong LLN for some sample values $\bsx_1,\dots,\bsx_n$.
Such a strong law was available for plain Monte Carlo but
not for RQMC, yet RQMC shows much better empirical results
in \cite{bala:etal:2019}.

An outline of this paper is as follows.
Section~\ref{sec:background}
presents the strong and weak laws of large numbers
referred to above as well as  MC and QMC and RQMC
sampling, making more precise some of the
conditions stated in this introduction.
It includes a lemma to show that functions
of bounded variation in the sense of Hardy and Krause
(the usual regularity assumption in QMC)
must also be Riemann integrable. That is either a new
result or one hard to find in the literature.
Section~\ref{sec:snets} defines the QMC
method known as digital nets whose
RQMC counterparts are called scrambled nets.
Section~\ref{sec:thellns} has the main result.
It is a strong law of large numbers for scrambled
net sampling. The integrand is assumed to
be square integrable.
The first new strong law is a form of consistency for scrambled net integration
as $n\to\infty$ through the set of values that can be written as $n=rb^m$ for
$r=1,\dots,R$ using some integers $m\ge0$,  $b\ge2$ and $R\ge1$.
While those are the best sample sizes to use for reasons given in that section,
we next extend the result to the ordinary limit as $n\to\infty$
through all integer values.
Section~\ref{sec:rieszthorin} replaces the assumption that $f^2$
be integrable by one that  $|f|^p$ have a finite integral
for some $p>1$. This result uses the
Riesz-Thorin interpolation theorem \cite{benn:shar:1988}.
Section~\ref{sec:discussion} provides some additional context and discussion,
including randomly shifted lattice versions of RQMC.

\section{Background on LLNs, QMC and RQMC}\label{sec:background}

We begin with the unit cube $[0,1]^d$ in dimension $d\ge1$.
For $p\ge1$, the space $L^p[0,1]^d$ consists of all measurable functions $f$ on $[0,1]^d$
for which $\Vert f\Vert_p = \bigl(\int_{[0,1]^d}|f(\bsx)|^p\rd\bsx\bigr)^{1/p} <\infty$.
We consider the problem of computing an estimate $\hat\mu$ of the integral
$\mu = \int_{[0,1]^d}f(\bsx)\rd\bsx$.
Here $\mu$ is the expected value of $f(\bsx)$
when $\bsx$ has the uniform distribution on $[0,1]^d$.
We write $\mu=\e(f(\bsx))$ for $\bsx\sim\dustd[0,1]^d$
and we use $\Pr(A)$ below to denote the probability of the event $A$.
Many problems that do not originate as integrals over $[0,1]^d$
have such a representation using transformations to generate non-uniformly distributed
random variables over the cube and other spaces
\cite{devr:1986}. We suppose that those transformations are subsumed into~$f$.
Also, while our theory works for genuinely random numbers, in practice one
ordinarily uses deterministic output of a random number generator that simulates
randomness.

The plain Monte Carlo (MC) method takes independent
$\bsx_i\sim\dustd[0,1]^d$ and estimates $\mu$ by
$
\hat\mu_n =\hat\mu_n^\mc = (1/n)\sum_{i=1}^nf(\bsx_i).
$
There are many more sophisticated Monte Carlo methods but when
we refer to Monte Carlo below we mean this simple one.

The weak law of large numbers (WLLN) implies that for any $\epsilon>0$,
\begin{align}\label{eq:mcwlln}
\lim_{n\to\infty} \Pr\bigl( |\hat\mu_n^\mc-\mu|>\epsilon\bigr)=0.
\end{align}
The strong law of large numbers (SLLN) implies that
\begin{align}\label{eq:mcslln}
\Pr\Bigl( \lim_{n\to\infty}\hat\mu_n^\mc=\mu\Bigr)=1
\end{align}
which we may write as
$\Pr( \lim\sup_{n\to\infty}|\hat\mu_n^\mc-\mu|>\epsilon)=0$
to parallel the WLLN.
Both the WLLN and SLLN hold for independent and identically distributed (IID)
random variables $f(\bsx_i)$ when $f\in L^1[0,1]^d$.
For proofs of these laws, see \cite[Chapter 2]{durr:2019}.
For an example of a sequence of independent random variables
that satisfies the WLLN but not the SLLN, let
$\hat\mu_n = \mu$ with probability $1-1/n$ and
$\hat\mu_n = \mu+1$ otherwise.

In QMC sampling, the $\bsx_i$ are constructed so that the discrete
distribution placing probability $1/n$
on each of $\bsx_1,\dots,\bsx_n$ (with repeated points counted multiple times)
is close to the continuous uniform distribution on $[0,1]^d$.
There are various ways, called discrepancies
\cite{chen:sriv:trav:2014}, to quantify the
distance between these discrete and continuous measures.
For a set $S\subset [0,1]^d$ define $\one\{\bsx\in S\}$
to be $1$ if $\bsx\in S$ and $0$ otherwise.
The most widely used discrepancy is the star discrepancy
$$
D^*_n  = D^*_n(\bsx_1,\dots,\bsx_n) = \sup_{\bsa\in[0,1]^d}
\biggl|
\frac1n\sum_{i=1}^n\one\{\bsx_i\in [\bszero,\bsa)\}
-\prod_{j=1}^da_j
\biggr|
$$
where $[\bszero,\bsa) =\{\bsx\in[0,1]^d\mid 0\le x_j < a_j,\, j=1,\dots,d\}$.

To keep this paper at a manageable length,
the relevant properties of QMC and RQMC methods are presented
but the details of their constructions are omitted.
For the latter, see \cite{dick:kuo:sloa:2013,dick:pill:2010,nied92,lecu:lemi:2002}
among others.

Because QMC is deterministic,
it has no analogue of the WLLN~\eqref{eq:mcwlln}.
There is an analogue of the SLLN~\eqref{eq:mcslln},
as follows.
Let $\hat\mu_n^\qmc = (1/n)\sum_{i=1}^nf(\bsx_i)$
where now the points $\bsx_i$ have been chosen to have
small discrepancy.
If $f$ is Riemann integrable and $D_n^*\to0$ then
\cite[p. 3]{kuip:nied:1974}
\begin{align}\label{eq:slln4qmc}
\lim_{n\to\infty}\hat\mu_n^\qmc = \mu
\end{align}
providing the QMC version of the SLLN~\eqref{eq:mcslln}.
There is a converse, where if $|\hat\mu_n-\mu|\to0$ whenever
$D_n^*\to0$, then $f$ must be Riemann integrable. See the references
and discussion in \cite{nied:1977}.  That is, QMC could fail
to be consistent when $f$ is not Riemann integrable.
Riemann integrable $f$ must also be bounded.

A better known result about QMC is the Koksma-Hlawka inequality
below which uses the notion of bounded variation.  Recall that
a differentiable function $f$ on $[0,1]$ has total variation
$V(f) = \int_0^1|f'(x)|\rd x$ and it is of bounded variation
for $V(f)<\infty$.
There are numerous generalizations of the total
variation for functions on the unit cube $[0,1]^d$ when  $d>1$ (see \cite{clar:adam:1933}).
Of those, the total variation
in the sense of Hardy and
Krause \cite{hard:1906,krau:1903:a},
denoted by $V_\hk(f)$, is the most useful one for QMC.
If $V_\hk(f)<\infty$, then we write $f\in \bvhk[0,1]^d$.
Although we don't need $f$ to have bounded variation to get the SLLN~\eqref{eq:slln4qmc}
for QMC, bounded variation gives us some information
on the rate of convergence, via the Koksma-Hlawka inequality
\begin{align}\label{eq:khineq}
|\hat\mu_n^\qmc-\mu|\le D_n^*(\bsx_1,\dots,\bsx_n)V_\hk(f)
\end{align}
(see \cite{hick:2014}).
Typical QMC constructions provide
infinite sequences $\bsx_i$ whose initial subsequences satisfy
$$
D_n^*(\bsx_1,\dots,\bsx_n) = O\Bigl( \frac{\log(n)^d}{n}\Bigr).
$$
Then $|\hat\mu^\qmc_n-\mu| = O(n^{-1+\epsilon})$
by~\eqref{eq:khineq} for any $\epsilon>0$.

The counterpart in MC to the Koksma-Hlawka inequality is that
\begin{align}\label{eq:rootn}
\e( (\hat\mu_n^\mc-\mu)^2)^{1/2} = n^{-1/2}\sigma(f)
\end{align}
when, for $\bsx\sim\dustd[0,1]^d$
the variance of $f(\bsx)$ is
 $\sigma^2=\sigma^2(f)=\e( (f(\bsx)-\mu)^2)<\infty$.
Where the rate for QMC comes after strengthening the regularity
requirement on $f$ from Riemann integrability to bounded variation, the rate for
MC comes about after strengthening the requirement from $f\in L^1[0,1]^d$ to $f\in L^2[0,1]^d$.
The MC counterpart~\eqref{eq:rootn}
 is exact while the QMC version~\eqref{eq:khineq} is an
extremely conservative upper bound, in that it covers even the
worst $f\in\bvhk[0,1]^d$ for any given $\bsx_1,\dots,\bsx_n$.

A Riemann integrable function is not necessarily
in BVHK.  For instance $f(\bsx) = \one\{\sum_{j=1}^dx_j \le 1\}$
is Riemann integrable but, for $d\ge2$, it is not in BVHK \cite{variation}.
A function in BVHK is necessarily Riemann integrable. This result
is hard to find in the literature.
It must almost certainly have been known to Hardy, Krause, Hobson and others
over a century ago, at least for $d=2$, which earlier work emphasized.
Here is a short proof based on some recent results.

\begin{lemma}
If $f\in\bvhk[0,1]^d$, then $f$ is also Riemann integrable.
\end{lemma}
\begin{proof}
If $f\in\bvhk[0,1]^d$ then
$f(\bsx) = f(\bszero) +f_+(\bsx)-f_-(\bsx)$
where $f_\pm$ are uniquely determined completely monotone
functions on $[0,1]^d$ with $f_\pm(\bszero)=0$ \cite[Theorem 2]{aist:dick:2015}.
Completely monotone functions are, {\sl a fortiori}, monotone.
Now both $f_\pm$ are bounded monotone functions
on $[0,1]^d$.
They are then  Riemann integrable
by the corollary in \cite{lavr:1993}.
\end{proof}

While QMC has a superior convergence rate to MC for $f\in\bvhk$,
MC has an advantage over QMC in that $\e( (\hat\mu^\mc-\mu)^2)=\sigma^2/n$
is simple to estimate from independent replicates, while
$D_n^*$ is very expensive to compute \cite{doer:gnew:wals:2014}
and $V_\hk(f)$ is much harder to estimate than $\mu$.
In a setting where attaining accuracy is important, it must  also
be important to estimate the attained accuracy.
RQMC methods, described next, are hybrids of MC and QMC that
support error estimation.

In RQMC \cite{lecu:lemi:2002,rtms} one starts with points $\bsa_1,\dots,\bsa_n\in[0,1]^d$
having a small star discrepancy and randomizes them to produce
points $\bsx_1,\dots,\bsx_n$. These points satisfy the following
two conditions: individually $\bsx_i\sim\dustd[0,1]^d$, and collectively,
$\bsx_1,\dots,\bsx_n$ have small star discrepancy.
The RQMC estimate of $\mu$  is $\hat\mu_n^\rqmc = (1/n)\sum_{i=1}^nf(\bsx_i)$.
From the uniformity of the points $\bsx_i$ we find that
$\e( \hat\mu_n^\rqmc)=\mu$.
Their small star discrepancy means that they are also QMC points
and so they inherit the accuracy properties of QMC.
To estimate the error, one takes several independent randomizations of $\bsa_i$
producing independent replicates of $\hat\mu_n^\rqmc$ whose sample
variance can be computed.


\section{Scrambled nets and sequences}\label{sec:snets}

In this section, we describe digital nets and sequences
and  scrambled versions of them.  Many authors reserve the term `digital'
to only mean points obtained from some certain specific
classes of algorithms.
Since the overwhelming majority of nets and sequences in use
are constructed with such algorithms, we lose little by this simplification.

Let $b\ge2$ be an integer base.
Let $\bsk = (k_1,\dots,k_d)$ for integers $k_j\ge0$
and $\bsc = (c_1,\dots,c_d)$ where $c_j\in\{0,1,\dots,b^{k_j}-1\}$.
Then the set
\begin{align}\label{eq:defelint}
E(\bsk,\bsc) =\prod_{j=1}^d \Bigl[ \frac{c_j}{b^{k_j}_j}, \frac{c_j+1}{b^{k_j}_j}\Bigr)
\end{align}
is called an elementary interval in base $b$.
It has volume $b^{-|\bsk|}$ where $|\bsk| =\sum_{j=1}^dk_j$.

\begin{definition}
For integers $m\ge t\ge0$, $b\ge2$ and $d\ge1$,
the points $\bsx_1,\dots,\bsx_n\in[0,1)^d$ for $n=b^m$
are a $(t,m,d)$-net in base $b$ if
$$
\sum_{i=1}^n\one\{\bsx_i\in E(\bsk,\bsc)\} = b^{m-|\bsk|}
$$
holds for every elementary interval $E(\bsk,\bsc)$ from~\eqref{eq:defelint}
with $|\bsk|\le m-t$.
\end{definition}

An elementary interval of volume $b^{-|\bsk|}$
should ideally contain $nb^{-|\bsk|}=b^{m-|\bsk|}$ points
from $\bsx_1,\dots,\bsx_n$.
In a $(t,m,d)$-net in base $b$, every elementary interval that should ideally
contain $b^t$ of the points does so.  For any given
$b$, $m$ and $d$, smaller $t$
imply finer equidistribution.  It is not always possible
to attain $t=0$.

\begin{definition}
For integers $t\ge0$, $b\ge2$ and $d\ge1$,
the points $\bsx_i\in[0,1)^d$ for $i\ge1$
are a $(t,d)$-sequence in base $b$ if
every subsequence of the form
$\bsx_{(r-1)b^m+1},\dots,\bsx_{rb^m}$ for
integers $m\ge t$ and $r\ge1$
is a $(t,m,d)$-net in base~$b$.
\end{definition}

The best available values of $t$ for nets and
sequences are recorded in the online resource MinT described
in~\cite{schu:schm:2009}, which also includes lower bounds.
The Sobol' sequences of \cite{sobol67}
are $(t,d)$-sequences in base $b=2$.
There are newer  versions of Sobol's sequence
with improved `direction numbers' in
\cite{joe:kuo:2008,sobo:asot:krei:kuch:2011}.
The Faure sequences \cite{faures} have $t=0$
but require that the base be a prime
number $b\ge d$.
Faure's construction was generalized to prime
powers $b\ge d$ in \cite{nied87}.
The best presently attainable values of $t$ for base $b=2$ are
in the Niederreiter-Xing sequences of \cite{niedxing96,niedxing96b}.

Randomizations of digital nets and sequences
operate by applying certain random permutations
to their base $b$ expansions.  For details, see
the survey in \cite{altscram}.
We will consider the `nested uniform' scramble from \cite{rtms}.

If $\bsa_1,\dots,\bsa_n$ is a $(t,m,d)$-net in base $b$
then after applying a nested uniform scramble, the
resulting points $\bsx_1,\dots,\bsx_n$ are a $(t,m,d)$-net
in  base $b$ with probability one \cite{rtms}.
If $\bsa_i$ for $i\ge1$ are a $(t,d)$-sequence in  base $b$
then after applying a nested uniform scramble, the resulting
points $\bsx_i$ for $i\ge1$ are a $(t,d)$-sequence in  base $b$
with probability one \cite{rtms}.
In either case, each resulting point satisfies $\bsx_i\sim\dustd[0,1]^d$.

If $f\in L^2[0,1]^d$ and $\hat\mu^{\rqmc}_n$ is based on
a nested uniform scramble of a $(t,d)$-sequence in base $b$
with sample sizes $n=b^k$ for integers $k\ge 0$,
then $\e( (\hat\mu^\rqmc_n-\mu)^2)=o(1/n)$ as $n\to\infty$.
It is thus asymptotically better than MC for any $f$.
For smooth enough $f$,
$\e( (\hat\mu^\rqmc_n-\mu)^2)=O( n^{-3+\epsilon})$ for any $\epsilon>0$.
See \cite{smoovar,localanti} for sufficient conditions.

The main result that we will use is as follows.
Let $f\in L^2[0,1]^d$ and write
$\sigma^2$ for the variance of $f(\bsx)$
when $\bsx\sim\dustd[0,1]^d$.
Then for a $(t,m,d)$-net in base $b$,
scrambled as in \cite{rtms}, we have
\begin{align}\label{eq:gammabound}
\e( (\hat\mu^\rqmc_n-\mu)^2)  \le \frac{\Gamma\sigma^2}n
\end{align}
for some $\Gamma<\infty$ \cite[Theorem 1]{snxs}.
That is, the RQMC estimate for these scrambled nets cannot
have more than $\Gamma$ times the mean squared
error that an MC estimate has.
The value of $\Gamma$
is found using some conservative upper bounds.  We can use
$\Gamma = b^t[(b+1)/(b-1)]^d$. If $t=0$, then we can take
$\Gamma = [b/(b-1)]^d$, and for $d=1$ we can take $\Gamma=b^t$.
The quantity $\Gamma$ arises as an upper bound on
an infinite set of `gain coefficients' relating the RQMC variance
to the MC variance for parts of a basis expansion of $f$.
The worst case bound $\sigma\sqrt{\Gamma/n}$
for the RQMC root mean squared error
does not contain the factor $\log(n)^d$ that makes
the QMC worst case error so large for large $d$
and $n$ of practical interest.

\section{RQMC laws of large numbers}\label{sec:thellns}

This section outlines some very simple LLNs for RQMC before
going on to prove two SLLN results for scrambled net integration
when $f\in L^2[0,1]^d$.
The first SLLN requires sample sizes to be of the form $rb^m$
for $1\le r\le R$ and $m\ge0$ where $b$ is the base of
those nets.  The second SLLN extends the first one to
include all integer sample sizes.

If $f\in\bvhk[0,1]^d$, then there is an SLLN
for RQMC from the Koksma-Hlawka inequality~\eqref{eq:khineq}
when $\Pr(\lim_{n\to\infty}D_n^*(\bsx_1,\dots,\bsx_n) =0)=1$.
More generally, for Riemann integrable~$f$
we get an SLLN for RQMC as an immediate
consequence of equation~\eqref{eq:slln4qmc}.

\begin{theorem}\label{thm:riemannslln}
Let $f:[0,1]^d\to\real$ be Riemann integrable.
For $i\ge1$, let $\bsx_i\in[0,1]^d$  be RQMC points
with $\Pr(\lim_{n\to\infty}D_n^*(\bsx_1,\dots,\bsx_n) =0)=1$.
Then $$\Pr\Bigl(\lim_{n\to\infty} \hat\mu_n^\rqmc=\mu\Bigr)=1.$$
\end{theorem}
\begin{proof}
From equation~\eqref{eq:slln4qmc},
$$\Pr\Bigl( \lim_{n\to\infty}\hat\mu^\rqmc_n = \mu\Bigr)
\ge \Pr\Bigl( \lim_{n\to\infty} D_n^*(\bsx_1,\dots,\bsx_n)=0\Bigr)=1.\qedhere$$
\end{proof}

Theorem~\ref{thm:riemannslln}  is not strong enough
for some important applications.
It does not cover integration problems
where the integrand $f$ is not in $\bvhk[0,1]^d$
including many where $f$ is not even Riemann integrable.
Integrands with jump discontinuities
or kinks (jumps in their gradient)
\cite{grie:kuo:sloa:2013,grie:kuo:sloa:2017,grie:kuo:leov:sloa:2018,he2015convergence}
commonly fail to be in BVHK
and integrands containing singularities
\cite{basu:owen:2018,hart:kain:2006,pointsingularity, sobo:1973}
are not even Riemann integrable.

Sobol' \cite{sobo:1973} noticed that some of his colleagues were using his QMC
points with apparent success on problems with integrable singularities and
then he initiated a theory in which QMC could be consistent provided the points $\bsx_i$
avoided the singularities in a suitable and problem specific way. Uniform random
points show no preference for the region near a singularity
no matter where it is
and this is enough to get consistent integral estimates
on some problems with integrable singularities
\cite{basu:owen:2018,haltavoid,pointsingularity}.

In those cases, we can easily get a WLLN, if the integrand is in $L^2$.
The usual results for RQMC
show that $\e( (\hat\mu_n^\rqmc-\mu)^2)\to0$
as $n\to\infty$ for $f\in L^2[0,1]^d$.
From that a WLLN follows by Chebychev's inequality.
A WLLN proves to be not quite enough for some
problems, so we seek an SLLN for scrambled net quadrature.

First we prove a strong law of large numbers
for sample sizes equal to $rb^m$ for $1\le r\le R$ and $b\ge2$
and $f\in L^2[0,1]^d$.
These are the best sample sizes to use in a $(t,d)$-net
with values $n=b^m$ being the best of those
because they are the smallest sample sizes to properly
balance elementary intervals of size $b^{t-m}$.

Sobol' \cite{sobo:1998} recommends using
sample sizes in a geometric progression such as $n_\ell=2^\ell$, not an arithmetic
one and there is a lengthier discussion of this point in \cite{owen2016constraint}.
To see informally how this works, suppose that
$|\hat\mu_n-\mu|\le An^{-1-\delta}$ for $\delta>1$ and $A>0$
while $|\hat\mu_n-\hat\mu_{n+1}|\ge B/n$.
The first is an instance of better than $1/n$
error and the second will be common because
$\hat\mu_{n+1} = \hat\mu_n(n/(n+1)) +f(\bsx_{n+1})/(n+1)$.
Then
$|\hat\mu_{n+1}-\mu|\ge
|\hat\mu_{n+1}-\hat\mu_n|-|\hat\mu_{n}-\mu|
$
and so for large $n$,  $\hat\mu_{n+1}$ will commonly
be worse than $\hat\mu_n$.  A rate like $n^{-1-\delta}$
can only be attained on geometrically spaced sample sizes $n$
under conditions in \cite{owen2016constraint}.

 \begin{theorem}\label{thm:slln4rqmc}
 Let $\bsx_1,\bsx_2,\dots$ be a $(t,d)$-sequence
 in base $b$, with gain coefficients no larger than $\Gamma<\infty$
 and randomized as in \cite{rtms}.
 Let $f\in L^2[0,1]^d$ with $\int_{[0,1]^d}f(\bsx)\rd\bsx=\mu$.
For an integer $R\ge1$, let $\cn=\{ rb^m\mid 1\le r\le R, m\ge 0\}$.
Then
$$
\Pr\Bigl( \lim_{\ell\to\infty} \hat\mu_{n_\ell}^\rqmc = \mu\Bigr) =1
$$
where $n_\ell$ for $\ell\ge1$ are the unique elements of $\cn$ arranged
in increasing order.
\end{theorem}

\begin{proof}
Pick any $\epsilon>0$.
Let $\sigma^2<\infty$ be the variance of $f(\bsx)$ for $\bsx\sim\dustd[0,1]^d$.
First we consider $n_\ell = rb^m$ for some $m\ge t$ and $1\le r\le R$.
Because $m\ge t$, the definition of a $(t,d)$-sequence
implies that
$$
\hat\mu_{n_\ell}^\rqmc = \frac1r\sum_{j=1}^r\hat\mu_{\ell,j}
$$
where each $\hat\mu_{\ell,j}$ is the average of $f$ over a
scrambled $(t,m,d)$-net in base $b$. We don't know the covariances
$\cov(\hat\mu_{\ell,j}, \hat\mu_{\ell,j'})$ but we can bound
them by assuming conservatively that the corresponding correlations are $1$.
Then
\begin{align*}
\var(\hat\mu_{n_\ell}^\rqmc)
&=\frac1{r^2}\sum_{j=1}^r\sum_{j'=1}^r
\cov(\hat\mu_{\ell,j}, \hat\mu_{\ell,j'})
\le \var(\hat\mu_{\ell,1})\le\frac{\Gamma\sigma^2}{n_\ell/r}.
\end{align*}
Next, by Chebychev's inequality,
$\Pr( |\hat\mu^\rqmc_{n_\ell}-\mu|\ge\epsilon)\le
r\Gamma\sigma^2/(n_\ell\epsilon^2)$.
Now
\begin{align}
\sum_{\ell=1}^\infty\Pr( |\hat\mu^\rqmc_{n_\ell}-\mu|\ge\epsilon)
&\le
\sum_{m=0}^\infty\sum_{r=1}^R
\Pr( |\hat\mu^\rqmc_{rb^m}-\mu|\ge\epsilon)\notag\\
&\le tR + \sum_{m=t}^\infty\sum_{r=1}^R\frac{\Gamma\sigma^2}{b^m\epsilon^2}.
\label{eq:preborelcantelli}
\end{align}
The first inequality arises because some sample sizes $n_\ell$ may
have more than one representation  of the form $rb^m$.
Because the sum~\eqref{eq:preborelcantelli}
is finite,
$$\Pr( |\hat\mu^\rqmc_{n_\ell}-\mu|\ge\epsilon\ \text{for infinitely many $\ell$})=0$$
by the Borel-Cantelli lemma \cite[Chapter 2]{durr:2019}.
Therefore $\Pr\bigl( \lim_{\ell\to\infty}\hat\mu^\rqmc_{n_\ell}=\mu\bigr)=1$.
\end{proof}

Next we extend this SLLN to a limit as $n\to\infty$ without
a restriction to geometrically spaced sample sizes.
While geometrically
spaced sample sizes should be used, it is interesting to verify this limit as well.
The proof method is adapted from the way that
Etemadi \cite{etem:1981}
extends an SLLN for pairwise independent and identically distributed
random variables from geometrically spaced sample sizes to all sample sizes.

\begin{theorem}\label{thm:slln4rqmcall}
 Let $\bsx_1,\bsx_2,\dots$ be a $(t,d)$-sequence
 in base $b$, with gain coefficients no larger than $\Gamma<\infty$
 and randomized as in \cite{rtms}.
Let $f\in L^2[0,1]^d$ with $\int_{[0,1]^d}f(\bsx)\rd\bsx=\mu$.
Then
$$
\Pr\Bigl( \lim_{n\to\infty} \hat\mu_{n}^\rqmc = \mu\Bigr) =1.
$$
\end{theorem}
\begin{proof}
First we suppose that $f(\bsx)\ge0$. This is no loss of generality because
$f(\bsx) = f_+(\bsx)-f_-(\bsx)$ where $f_+(\bsx)=\max(f(\bsx),0)$
and $f_-(\bsx) = \max(-f(\bsx),0)$. If $f\in L^2[0,1]^d$ then both $f_\pm\in L^2[0,1]^d$
and an SLLN for $f_\pm$ would imply one for $f$.

Because $f(\bsx_i)\ge0$, we know that
$T(n)\equiv\sum_{i=1}^nf(\bsx_i)$ is nondecreasing in $n$.
Choose $R=b^k$ for $k>1$ and let $\cn = \cn(R) = \{ rb^m\mid 1\le r\le R, m\ge0\}$.
For any integer $n\ge1$ define
$\bar n = \bar n(n) =
\min\{ \nu\in\cn\mid \nu\ge n\}$
and
$\underline n =\underline n(n)
= \max\{ \nu\in\cn\mid \nu\le n\}$. Monotonicity of $T(n)$
combined with $\hat\mu_n^\rqmc = T(n)/n$ gives
$$
\frac{\underline n(n)}{n}
\hat\mu_{\underline n}^\rqmc
\le \hat\mu_n^\rqmc
\le\frac{\bar n(n)}n
\hat\mu_{\bar n}^\rqmc.
$$
By Theorem~\ref{thm:slln4rqmc},
$\Pr( \limsup_{n\to\infty} \hat\mu^\rqmc_{\bar n}=\mu)=1$
and
$\Pr( \liminf_{n\to\infty} \hat\mu^\rqmc_{\underline n}=\mu)=1$.
What remains is to bound $\bar n/n$ and $\underline n/n$.

We can suppose that $n>b^k$.
The base $b$ expansion of $n$ is
$\sum_{\ell=0}^La_\ell b^\ell$ where $a_\ell = a_\ell(n)\in\{0,1,\dots,b-1\}$
and $L=L(n) = 1+\lfloor \log_b(n)\rfloor$
is the smallest number of base $b$ digits required to write $n$.
Choosing $m=L-k+1$ we know that
$\underline n\ge \nu=rb^m$ for
$r=\sum_{s=0}^{L-m}a_{m+s}b^s\le b^k= R$.
As a result
$$
\frac{\underline n(n)}{n}
\ge \frac{\sum_{\ell=L-k+1}^La_\ell b^\ell}{\sum_{\ell=0}^La_\ell b^\ell}
\ge \frac{\sum_{\ell=L-k+1}^La_\ell b^\ell}{b^{L-k+1}+\sum_{\ell=L-k+1}^La_\ell b^\ell}
\ge \frac{b^L}{b^{L-k+1}+b^L}.
$$
It follows that
$$\Pr\Bigl( \liminf_{n\to\infty} \hat\mu^\rqmc_{n}\ge(1+b^{1-k})^{-1}\mu\Bigr)=1$$
and since we may choose $k$ as large as we like,
$\Pr( \liminf_{n\to\infty} \hat\mu^\rqmc_{n}\ge\mu)=1$.
Similarly, if $n=A_Lb^L$ then $n\in\cn$ and we may take
$\bar n = n$.  Otherwise,  $\bar n\le\nu+b^m=(r+1)b^m$
with $r+1\le R$ and then
$\Pr( \liminf_{n\to\infty} \hat\mu^\rqmc_{n}\le\mu)=1$.
\end{proof}

\section{An SLLN without square integrability}\label{sec:rieszthorin}

The SLLN for Monte Carlo only requires that $f\in L^1[0,1]^d$.
The results in Section~\ref{sec:thellns} for RQMC require
the much stronger condition that $f\in L^2[0,1]^d$.
In this section, we narrow the gap by proving an SLLN for
scrambled nets when $f\in L^{p}[0,1]^d$
for some $p>1$.

The proof is based on the Riesz-Thorin interpolation theorem
from \cite[Chapter 4]{benn:shar:1988}.
Let $\ce$ be the operator that takes an integrand $f$ and returns
the integration error
$$
\hat\mu^\rqmc_n-\mu = \frac1n\sum_{i=1}^nf(\bsx_i)-\mu.
$$
The integration error is a function of $\bsx_1,\dots,\bsx_n\in[0,1]^d$.
Together these belong to $[0,1]^{dn}$.
Let $\Omega$ be the set $[0,1]^{dn}$ equipped with the distribution
induced by the scrambled net randomization producing
$\bsx_1,\dots,\bsx_n$.
If $n=b^m$, then  $\ce$ is a bounded linear operator from $L^2[0,1]^d$
to $L^2(\Omega)$.
The norm of $\ce$ is
$$
\Vert\ce\Vert_{L^2[0,1]^d\to L^2(\Omega)}
= \sup_{\Vert f\Vert_2\le1}\bigl( \e(\hat\mu^\rqmc_n-\mu)^2\bigr)^{1/2}
\le \sqrt{\Gamma/n}.
$$

The operator $\ce$ is also a bounded linear operator from
$L^1[0,1]^d$ to $L^1(\Omega)$.
Here the norm is
$$
\Vert\ce\Vert_{L^1[0,1]^d\to L^1(\Omega)}
= \sup_{\Vert f\Vert_1\le1}
\e(|\hat\mu^\rqmc_n-\mu|)
\le \sup_{\Vert f\Vert_1\le1}|\mu(f)| + \int_{[0,1]^d}|f(\bsx)|\rd\bsx\le 2.
$$
By the Riesz-Thorin theorem below, $\ce$ is also  a bounded linear
operator from $L^p[0,1]$ to $L^p(\Omega)$ for any $p$
with $1\le p\le 2$.

\begin{theorem}[Riesz-Thorin]
For $1\le q_1\le q_2<\infty$ and $\theta\in[0,1]$, let $p\ge1$ satisfy
$$
\frac1p = \frac{1-\theta}{q_1} +\frac\theta{q_2}.
$$
For probability spaces $\Theta_1$ and $\Theta_2$,
let $\ct$ be a linear operator from
$L^{q_1}(\Theta_1)$ to $L^{q_1}(\Theta_2)$
and at the same time a linear operator
from $L^{q_2}(\Theta_1)$ to $L^{q_2}(\Theta_2)$ satisfying
$$
\Vert \ct\Vert_{L^{q_1}(\Theta_1)\to L^{q_1}(\Theta_2)}\le M_1
\quad\text{and}\quad
\Vert \ct\Vert_{L^{q_2}(\Theta_1)\to L^{q_2}(\Theta_2)}\le M_2.
$$
Then $\ct$ is a linear operator from
$L^{p}(\Theta_1)$ to $L^{p}(\Theta_2)$ satisfying
$$\Vert \ct\Vert_{L^{p}(\Theta_1)\to L^{p}(\Theta_2)}\le
M_1^{1-\theta}M_2^\theta.
$$
\end{theorem}
\begin{proof}
This is a special case of Theorem 2.2(b) in \cite{benn:shar:1988}.
\end{proof}

Because $1/p$ is a convex combination of $1/q_1$
and $1/q_2$ we must have $q_1\le p\le q_2$.
Our interest is in $q_1=1$ and $q_2=2$ and $1\le p\le 2$.
The following corollary handles that case.

\begin{corollary}\label{cor:oneandtwo}
Let $\ct$ be a linear operator from $L^1(\Theta_1)$ to $L^1(\Theta_2)$
and at the same time from $L^2(\Theta_1)$ to $L^2(\Theta_2)$ with
$$
\Vert \ct\Vert_{L^{1}(\Theta_1)\to L^{1}(\Theta_2)}\le M_1
\quad\text{and}\quad
\Vert \ct\Vert_{L^{2}(\Theta_1)\to L^{2}(\Theta_2)}\le M_2.
$$
Then for $1\le p\le 2$,
$$
\Vert \ct\Vert_{L^{p}(\Theta_1)\to L^{p}(\Theta_2)}\le M_1^{(2-p)/p}M_2^{2(p-1)/p}.
$$
\end{corollary}

Now we are ready to use the Riesz-Thorin theorem to get
an SLLN.  The operator $\ct$ will be the RQMC error $\ce$,
the space $\Theta_1$ will be $[0,1]^d$ under the uniform
distribution and the space $\Theta_2$ will be $[0,1]^{nd}$
under the distribution induced by the RQMC points $\bsx_1,\dots,\bsx_n$.

\begin{theorem}\label{thm:slln4rqmcallrt}
 Let $\bsx_1,\bsx_2,\dots$ be a $(t,d)$-sequence
 in base $b$, with gain coefficients no larger than $\Gamma<\infty$
 and randomized as in \cite{rtms}.
For $p>1$, let $f\in L^p[0,1]^d$ with $\int_{[0,1]^d}f(\bsx)\rd\bsx=\mu$.
Then
$$
\Pr\Bigl( \lim_{n\to\infty} \hat\mu_{n}^\rqmc = \mu\Bigr) =1.
$$
\end{theorem}
\begin{proof}
For $p\ge2$, the conclusion follows
from Theorem~\ref{thm:slln4rqmcall} and so we assume
now that $1<p<2$.  Choose any $\epsilon>0$
and suppose that $n=rb^m$ for $1\le r\le R<\infty$ and $m\ge0$.
The error operator $\ce$ for this $n$
satisfies $\Vert\ce\Vert_{L^1} \le 2$
and $\Vert\ce\Vert_{L^2} \le (r\Gamma/n)^{1/2}$.
Taking $\ct=\ce$ in Corollary~\ref{cor:oneandtwo},
$$\sup_{\Vert f\Vert_p\le 1} \bigl( \e(|\hat\mu^\rqmc_n-\mu|^p\bigr)^{1/p}
\le 2^{(2-p)/p}\Bigl(\frac{r\Gamma}{n}\Bigr)^{(p-1)/p}$$
from which
$\e(|\hat\mu^\rqmc_n-\mu|^p)
\le 2^{2-p} (r\Gamma/n)^{p-1}$ and then
$$
\Pr( |\hat\mu_n^\rqmc-\mu|>\epsilon)
\le 2^{2-p}\epsilon^{-p} (r\Gamma)^{p-1}\Vert f\Vert_p^p n^{1-p}.
$$
This probability has a finite sum over $r=1,\dots,R$ and $m\ge0$
and so
$$\Pr\Bigl(\lim_{n\to\infty}\hat\mu^\rqmc_n=\mu\Bigr)=1$$
when the limit is over $n\in\{rb^m\mid 1\le r\le R, m\ge0\}$.
We have thus established a version of Theorem~\ref{thm:slln4rqmc}
for $p>1$ and the extension to the unrestricted limit as $n\to\infty$
uses the same argument as Theorem~\ref{thm:slln4rqmcall}.
\end{proof}

The Riesz-Thorin theorem has been previously used to
bound $p$'th moments in similar problems.
See for instance
 \cite{He94,rudolf2015error,kunsch2019solvable}.

\section{Discussion}\label{sec:discussion}

We have proved a strong law of large numbers for
scrambled digital net integration,
first for geometrically spaced sample sizes and a square
integrable integrand, then removing the geometric
spacing assumption and finally, reducing the squared
integrability condition to $\e(|f(\bsx)|^p)<\infty$
for some $p>1$.
It is interesting that this strong law for $p>1$
is obtained before an equally general weak law was found.

There are other ways to scramble digital nets and sequences.
The linear scrambles of \cite{mato:1998:2} require less
space than the nested uniform scramble.
They have the same mean squared discrepancy
as the nested uniform scramble \cite{HicYue00}
and so they might also satisfy an SLLN.
A digital shift \cite{lecu:lemi:2002,altscram} does not produce the same variance
as the nested uniform scramble and it does not
satisfy the critically important bound~\eqref{eq:gammabound}
on gain coefficients, so the methods used here would
not provide an SLLN for it.
The nested uniform scramble
is the only one for which central limit theorems have been proved
\cite{basu2017asymptotic,loh:2003}.

A second major family of RQMC methods has been constructed
from lattice rules \cite{sloa:joe:1994}.  Points $\bsa_1,\dots,\bsa_n$
on a lattice in $[0,1]^d$ are randomized into $\bsx_i = \bsa_i+\bsu \mod 1$,
for $\bsu\sim\dustd[0,1]^d$.
That is, they are shifted with wraparound
in what is known as a Cranley-Patterson rotation  \cite{cranpat76}.
Then the estimate of $\mu$ is $\hat\mu^\rlat_n = (1/n)\sum_{i=1}^nf(\bsx_i)$.
For an extensible version of shifted lattice rules, see \cite{hick:hong:lecu:lemi:2000}.
The Cranley-Patterson rotation does not provide a $\Gamma$ bound
like~\eqref{eq:gammabound} because there are functions $f\in L^2[0,1]^d$
with $\var(\hat\mu^\rlat_n)=\sigma^2(f)$ \cite{lecu:lemi:2002}, and so a proof of an SLLN
for this form of RQMC would require a different approach.
The fact that $\var(\hat\mu^\rlat_n)=\sigma^2(f)$ is possible does
not provide a counter-example to an SLLN because this equality might only
hold for a finite number of $n_\ell$ in the infinite sequence.
Given a class of functions $\cf$ with $\var(\hat\mu^\rlat_{n_\ell})\le B\sigma^2(f)/n_\ell$
for all $f\in\cf$, all $\ell\ge1$, and some $B<\infty$, we get an SLLN for $f\in\cf$
if $\sum_{\ell=1}^\infty 1/n_\ell<\infty$. Some such bounds $B$ for randomly
shifted lattices appear in \cite{lecu:lemi:2002}
though they hold for specific $n_\ell$ not necessarily an infinite sequence of them.

\section*{Acknowledgments}
We thank Max Balandat for posing the problem of finding an SLLN
for randomized QMC. Thanks also to Ektan Bakshy, Wei-Liem Loh and
Fred Hickernell for discussions.
This work was supported by grant IIS-1837931 from the U.S. National Science
Foundation.

\bibliographystyle{plain}
\bibliography{qmc}

\end{document}